\documentclass[11pt,a4paper]{article}
\usepackage[dvipdfmx]{graphicx}
\usepackage[abs]{overpic} 
\usepackage{amsfonts}
\usepackage{amsmath}
\usepackage{amsthm}
\usepackage{txfonts}
\usepackage{enumerate}
\usepackage{mathrsfs}
\usepackage{pifont} \usepackage[all]{xy}
\usepackage{graphicx,color}
\usepackage{comment}

\newtheorem{thm}{Theorem}[section]    
\newtheorem{lem}[thm]{Lemma}          
\newtheorem{prp}{Proposition}
\newtheorem{inpA}{Proposition}

\newtheorem{inpB}{Proposition}

\newtheorem{inpC}{Proposition}

\newtheorem{inpD}{Proposition}

\newtheorem{intt}{Theorem}

\theoremstyle{definition}
\newtheorem{defn}[thm]{Definition}    
\newtheorem*{rem}{Remark}             
\newtheorem{df}{Definition}[section]
\newtheorem{exam}[df]{Example}
\newcommand{\Map}{\mathop{\mathrm{Map}}\nolimits}

\begin{document}
\begin{center}
	\Large{Loop homology of some global quotient orbifolds}
\end{center}
\begin{center}{Yasuhiko Asao}\end{center}

\begin{abstract}
We determine the ring structure of the loop homology of some global quotient orbifolds. We can compute by our theorem the loop homology ring with suitable coefficients of the global quotient orbifolds of the form $[M/G]$ for $M$ being some kinds of homogeneous manifolds, and $G$ being a finte subgroup of a path connected topological group $\mathcal{G}$ acting on $M$. It is shown that these homology rings split into the tensor product of the loop homology ring of the manifold $\mathbb{H}_{*}(LM)$ and that of the classifying space of the finite group,which coincides with the center of the group ring $Z(k[G])$. 
\end{abstract}

\section{Introduction}
The free loop space of a topological space $X$ is a space of the continuous maps from the circle $S^{1}$ to $X$, 
\begin{equation}
LX = \Map (S^{1},X)
\end{equation}
with the compact-open topology.
The loop homology of $X$ is the homology of free loop space $H_{*}(LX)$. In the 1990's Moira Chas and Dennis Sullivan discovered a product $\circ$ on the loop homology of a closed oriented smooth manifold $H_{*}(LM)$
\begin{equation}
\circ : H_{p}(LM) \otimes H_{q}(LM) \longrightarrow H_{p+q-\dim M}(LM)
\end{equation}
called the loop product which is a mixture of the intersection product of a manifold and the concatenation operation identifying $S^{1}\vee S^{1}$ with the image of a map from $S^{1}$ (Chas, Sullivan \cite{CS}). They also showed in \cite{CS} that the product defines a ring structure and a kind of Lie algebra structure called the Batalin-Vilkovisky algebra (BV- algebra) structure on the homology. These rich algebraic structures are called \textit{string topology} of closed oriented manifolds. The string topology of a manifold is related to several areas of mathematics including mathematical physics through the tools of algebraic topology.

The string topology of wider class of spaces has also developed by several authors. 
\begin{itemize}
\item For classifyng spaces of connected compact Lie groups, the existence of the loop product and the BV-structure on its homology is proved by Chataur and Menichi in \cite{CM}.
\item For fiberwise monoids including the adjoint bundle of principal bundles, the loop product is constructed by Gruher and Salvatore in \cite{GS}.
\item And more generaly, for the Borel construction of a smooth manifold with a smooth action of a compact Lie group, the loop product is constructed by Kaji and Tene in \cite{KT}.
\item In the 2000's, Lupercio, Uribe, and Xicot\'{e}ncatl defined the loop homology of global quotient orbifold which is an orbifold of the form $[M/G]$ for $M$ being a smooth manifold and $G$ being a finite group acting smoothly on $M$, as the loop homology of the Borel construntion $M\times_{G}EG$. They discovered a product
\begin{equation}
\circ : H_{p}(L(M\times_{G} EG)) \otimes H_{q}(L(M\times_{G} EG)) \longrightarrow H_{p+q-\dim M}(L(M\times_{G}EG))
\end{equation}
on the loop homology of a global quotient orbifold $H_{*}(L(M\times_{G} EG))$, and showed that the product defines on $H_{*}(L(M\times_{G} EG))$ a BV-algebra structure  \cite{LUX}. They coined this structure with the name \textit{orbifold string topology}. 
\end{itemize}
In spite of these interesting structural discoveries, concrete computations of loop homology are achieved for only a few kinds of classes of manifolds. In order to accomodate the change in grading, we define $\mathbb{H}_{*}(LM) = H_{*+\dim M}(LM) $.
\begin{itemize}
\item For a compact Lie group $\Gamma$, there is a homeomorphism $L\Gamma \cong \Omega \Gamma \times \Gamma $, and a loop homology ring isomorphism $\mathbb{H}_{*}(L\Gamma) \cong H_{*}(\Omega \Gamma) \otimes \mathbb{H}_{*}(\Gamma) $ . Furthermore, the BV--algebra structure with coeffcients in $\mathbb{Q}$ and $\mathbb{Z}_{2}$ of any compact Lie group $G$ is determined by Hepworth in \cite{He1}.
\item For spheres $S^{n}$ and complex projective spaces $CP^{n}$, the ring structure with coefficient in $\mathbb{Z}$ is determined by Cohen, Jones, and Yan in \cite{CJY} by constructing the spectral sequence converging to the loop homology.
\item The BV-algebra structure for $CP^{n}$ with coefficients in $\mathbb{Z}$  is determined by Hepworth in \cite{He2}. 
\item For complex Stiefel manifolds $SU(n)/SU(k)$ including odd dimensional spheres $S^{2n+1}$, the BV- structure is determined by Tamanoi in \cite{T}.
\item  For arbitary spheres, Menichi determines the BV-structure for it in \cite{Me1} by using the Hochschild cohomology.
\item For the aspherical manifold $K(\pi ,1)$, the BV-struture is determined by Vaintrob in \cite{Vai} by establishing an isomrphism between the loop homology $\mathbb{H}_{*}(K(\pi,1))$ and the Hochschild cohomology $HH^{*}(\mathbb{Z}[\pi];\mathbb{Z}[\pi])$, and another proof is obtained by Kupers in \cite{Ku}.
\end{itemize}
The purpose of this paper is to determine the ring structure of some global quotient orbifolds by using the method of the orbifold string topology. We can compute by our theorem the loop homology ring with suitable coefficients of the global quotient orbifolds of the form $[M/G]$ for $M$ being a homogeneous manifold of a connected Lie group $\mathcal{G}$, and $G$ being a finte subgroup of $\mathcal{G}$. 

Now we briefly review a part of the work of Lupercio-Uribe-Xicot\'{e}ncatl in \cite{LUX}. For simplicity, we denote $\Map(S^{1},M\times_{G} EG)$ as $L[M/G]$. In \cite{LUX}, the $\it{loop \ orbifold}$  of the global quotient orbifold $[M/G]$ is defined as the groupoid $P_{G}M \rtimes G$, and they show that its Borel construction $P_{G}M\times_{G} EG$ is weak homotopy equivalent to the free loop space $L(M\times_{G}EG)$. To determine the loop homology ring of the lens space $S^{2n+1}/\mathbb{Z}_{p}$, they constructed a \textit{non $G$-equivariant} homotopy equivalence
\begin{equation}
P_{\mathbb{Z}_{p}}S^{2n+1} \simeq \displaystyle \coprod _{g\in \mathbb{Z}_{p}} LS^{2n+1}
\end{equation}
by using the fact that the action of $\mathbb{Z}_{p}$ extends to an $S^{1}$ action. The following is Proposition \ref{prp4} in this paper.
\begin{inpA}[\cite{LUX} ] 
	Let $M$ be a closed oriented manifold, and $G$ be a finite subgroup of a path connected group $\mathcal{G}$ acting continuously on $M$. Then for each $g \in G$, there exists a homotopy equivalence
	\begin{equation}
	P_{g}X \simeq LM  .
	\end{equation}
\end{inpA}
 We prove that this homotopy equivalence can be extended to wider class of spaces. Furthermore, we find a condition so that the above homotopy equivalence is $G$-$\it{equivariant \ at \ homology \ level}$. 

Now we state our main theorem. Let $\mathcal{G}$ be a path connected group acting continuously on $M$, and $G$ a finite subgroup of $\mathcal{G}$ acting smoothly on $M$. Then there is the map $\Phi : \Omega\mathcal{G} \times LM \longrightarrow LM$, with $\Phi (a,l) = (t \mapsto a(t)\cdot l(t))$, and this map induces an action of the Pontrjagin ring $H_{*}(\Omega \mathcal{G})$, namely $\Phi_{*} : H_{*}(\Omega \mathcal{G})\otimes H_{*}(LM) \longrightarrow H_{*}(LM)$. If the action $\Phi_{*}$ of $H_{0}(\Omega\mathcal{G};k)$ on $H_{*}(LM;k)$ satisfies the equation $\Phi_{*}(x) = x$, for any $x \in H_{*}(LM)$, we call the action $\Phi_{*}$  \textit{trivial}. Then we prove the following, which is Proposition \ref{prp} in this paper.
	\begin{inpB}
Assume that the action $\Phi_{*}$ is trivial with coefficient in $k$. Then the direct sum of the homotopy equivalence of the above proposition
	\[
	\displaystyle \coprod _{g \in G} P_{g}M \simeq \displaystyle \coprod _{g \in G}LM
	\]
	is $G$-equivariant at homology level with coefficients in $k$.
\end{inpB}
By using this homotopy equivalence, we can compute the loop homology of certain class of orbifolds. Our main theorem is the following, Theorem \ref{thm} in this paper. We denote the order of a group $G$ by $|G|$.
\begin{intt}
	Let $\mathcal{G}$ be a path connected topological group acting continuously on an oriented closed manifold $M$, $G$ be its finite subgroup, and $k$ be a field whose characteristic is coprime to $|G|$.
	If the action $\Phi_{*}$ is trivial with coefficient in $k$, then there exists an isomorphism as $k$-algebras
	\begin{equation}\label{iso}
	\mathbb{H}_{*}(L[M/G];k) \cong \mathbb{H}_{*}(LM;k) \otimes Z(k[G]) ,
	\end{equation}
	where $Z(k[G])$ denotes the center of the group ring $k[G]$. 
	\end{intt}
	We show that the following are the necessary conditions for the action $\Phi_{*}$ being trivial. They are Propositon \ref{ex1} and Proposition \ref{ex2} in this paper.
\begin{inpD}
If $\mathcal{G}$ is simply connected, then the action $\Phi_{*}$ is trivial for any field $k$.
	\end{inpD}
\begin{inpD}
If $H_{\dim M}(LM;k) = k$, then the action $\Phi_{*}$ is trivial for any pair $(\mathcal{G}, G)$.
	\end{inpD}
	\begin{inpC}
If the conditions
	\begin{itemize}
	\item[(i)] $M$ is simply connected,
	\item[(ii)] $|\pi_{1}\mathcal{G}| < \infty$,
	\item[(iii)] the homomorphism $H_{*}(\Omega M;k) \longrightarrow H_{*}(LM;k)$ induced by the inclusion $\Omega M \longrightarrow LM$ is injective, namely the free loop fibration $\Omega M \to LM \to M$ is Totally Non-Cohomologous to Zero (TNCZ) with coefficient in the field $k$,
	\item[(iv)] the characteristic of $k$ is coprime to $|\pi_{1}\mathcal{G}|$,
	\end{itemize}
	are satisfied, then the action $\Phi_{*}$ is trivial.
	\end{inpC}
	The organization of this paper is as follows. After this introduction in section 1, we briefly review the string topology first developed in \cite{CS} and the orbifold string topology constructed in \cite{LUX} in section 2. In section 3, we show some propositions necessary for the proof of the main theorem. In section 4, we prove the theorem first for vector spaces and second for algebras. Finally in section 5, we compute concrete examples by applying our theorem.
\section*{Acknowledgment}
The author would like to express his sincere gratitude to his advisor Professor Toshitake Kohno for all advice and encouragement. He also would like to show his appreciation to Takahito Naito for fruitful communication.

\section{Preliminalies for String topology}
In this section, we briefly review the loop product in string topology. 
\subsection*{Loop product}
Let $M$ be a smooth closed oriented manifold, and $LM = \Map(S^{1}, M)$ be the free loop space of $M$, the space of piecewise smooth maps from $S^{1}$ to $M$ with compact open topology. Then we have the following pullback diagram
\begin{equation}
	\xymatrix{
	\Map(S^{1}\vee S^{1},M) \ar@{=}[r] & LM\times_{M}LM \ar[r]^{\tilde{\iota}} \ar[d]^{ev_{1/2}} &  LM\times LM \ar[d]^{ev\times ev} \\
	 & M \ar[r]^{\iota}  & M\times M ,
}
\end{equation}
where $\iota$ denotes the diagonal embbeding, and $ev_{t}$ denotes the evaluation map with $ev_{t}(l) = l(t)$.
Then we can consider above $\tilde{\iota}$ as codimension $n$ embbeding of the infinite dimensional manifold, and we have the generalized Pontrjagin-Thom map due to R.Cohen and Klein \cite{CK}
\begin{equation}
\tilde{\iota}_{!*} : H_{*}(LM\times LM) \longrightarrow H_{*}((LM\times_{M} LM)^{ev_{1/2}^{*}\nu_{\iota}})  ,
\end{equation}
where $\nu_{\iota}$ denotes the normal bundle of the embbeding $\iota$, and $(LM\times_{M} LM)^{ev_{1/2}^{*}\nu_{\iota}}$ denotes the Thom space of the vector bundle $ev_{1/2}^{*}\nu_{\iota}$. The similar but more homotopy theoretic construction of this umkehr map is considered in \cite{KT}.
The loop product is formulated in \cite{CJ} as the composition of maps $\tilde{\iota}_{!*}$, the Thom isomorphism
\begin{equation}
\xymatrix{
H_{*}((LM\times_{M} LM)^{ev_{1/2}^{*}\nu_{\iota}}) \ar[r]^(0.53){\cong} & H_{*\dim M}(LM\times_{M}LM)  ,
}
\end{equation}
and the concatenating map $\gamma : LM\times_{M}LM \longrightarrow LM $ with
\begin{equation}\label{con1}
\gamma(l_{1},l_{2}) = l_{1}*l_{2} :=
\begin{cases}
	l_{1}(2t) & (0 \leq t \leq 1/2) ,\\
	l_{2}(2t-1) & (1/2 \leq t \leq 1) .
	\end{cases}
\end{equation}
\begin{defn}[\cite{CJ}]\label{def1}
The loop product is defined by the following sequence of compositons
\begin{align*}
H_{p}(LM)\otimes H_{q}(LM) \xrightarrow{\times} \  H_{p+q}(LM\times LM) \xrightarrow{e_{!*}} \  H_{p+q}((LM\times_{M} LM)^{ev_{1/2}^{*}\nu_{\iota}}) \\
 \xrightarrow{Thom \ isom.} \  H_{p+q-n}(LM\times_{M}LM) \xrightarrow{\gamma_{*}} \  H_{p+q-n}(LM)  .
\end{align*}
\end{defn}
In order to accomodate the change in grading, we define $\mathbb{H}_{*}(LM) = H_{*+ \dim M}(LM)$.
Chas and Sullivan prove the following in \cite{CS}.
\begin{thm}[Chas-Sullivan \cite{CS}]\label{thm0}
The loop product makes $\mathbb{H}_{*}(LM)$ an associative graded commutative algebra.
\end{thm}
\begin{rem}
Cohen - Jones gives in \cite{CJ} an operadic proof of Thorem \ref{thm0}, and Tamanoi gives in \cite{T2} more homotopy teoretic one.
\end{rem}
\subsection*{Orbifold}
In this section, we review the basic definitions and properties on orbifolds that we will use in this paper. Following Moerdijk \cite{Mo}, we use the groupoid notion of an orbifold. For more detail, see \cite{ALR}, \cite{Mo}.
\begin{defn}
A \emph{groupoid} is a category whose morphisms are all invertible. In other words, a groupoid ${\sf G}$ is a pair of sets $({\sf G_{0}},{\sf G_{1}})$ with the structure maps
\begin{description}
\item[source and target] $s,t : {\sf G_{1}} \longrightarrow {\sf G_{0}} $
\item[identity] $e : {\sf G_{0}} \longrightarrow {\sf G_{1}}$
\item[inverse] $i : {\sf G_{1}} \longrightarrow {\sf G_{1}}$
\item[composition] $m:{\sf G_{1}} \times_{{\sf G_{0}}} {\sf G_{1}} = \{(a, b) \in {\sf G_{1}\times {\sf G_{1}}} | t(a) = s(b)\}\longrightarrow {\sf G_{1}} $, 
\end{description}
satisfying suitable compatibilities.
\end{defn}
We note here some technical terms on groupoids. A \emph{Lie groupoid} is a groupoid $({\sf G_{0}}, {\sf G_{1}})$ such that ${\sf G_{0}}$ and ${\sf G_{1}}$ both have the structure of a smooth manifold, and the structure maps are all smooth. We will also require that $s,t$ are submersions. An \emph{isotropy group} of a groupoid $({\sf G_{0}}, {\sf G_{1}})$ at $x \in \sf{G_{0}}$ is the group $G_{x} = \{ f \in {\sf G_{1}}| s(f) = t(f) = x \}$. A Lie groupoid $({\sf G_{0}}, {\sf G_{1}})$ is said to be a \emph{proper groupoid} if the map $(s,t) : {\sf G_{1}} \longrightarrow {\sf G_{0}} \times {\sf G_{0}}$ is proper. A Lie groupoid $({\sf G_{0}}, {\sf G_{1}})$ is said to be a \emph{foliation groupoid} if the isotropy group $G_{x}$ is dicrete for each $x \in {\sf G_{0}}$.
\begin{defn}
A groupoid is said to be an \emph{orbifold groupoid} if it is a proper foliation Lie groupoid.
\end{defn}
\begin{exam}
Let $S$ be a set, and $\mathcal{G}$ be a group acting on $S$. Then $(S, S\times \mathcal{G})$ is a groupoid with structure maps
\begin{itemize}
\item $s(x, g) = x$,
\item $t(x, g) = gx$,
\item $e(x) = (x, id_{\mathcal{G}})$,
\item $i(x, g) = (gx, g^{-1})$,
\item$m((x, g), (gx, h)) = (x, hg)$  ,
\end{itemize} 
for any $x ,y \in S$ and $g, h \in \mathcal{G}$. We call this groupoid \emph{action groupoid}, and denote by $S \rtimes \mathcal{G}$.
\end{exam}
\begin{exam}
Let $M$ be a smooth manifold, and $\Gamma$ be a compact Lie group acting smoothly on $M$. If the isotropy group $\Gamma_{x}$ is finite for each $x \in M$, then the action groupoid $M \rtimes \Gamma$ is an orbifold groupoid.
\end{exam}
Orbifolds are defined as follows by using the notion of groupoids.
\begin{defn}
A Morita equivalent class $[{\sf G}]$ of orbifold groupoids is called an \emph{orbifold}.
\end{defn}
\begin{defn}
An orbifold ${\sf X}$ is called a \emph{global quotient orbifold} if ${\sf X}$ has a representation of an action groupoid $M \rtimes G$, where $M$ is a smooth manifold, and $G$ is a finite group. We denote ${\sf X} = [M/G]$.
\end{defn}
The following are fundamental properties of orbifolds. See for example \cite{ALR} for the proof.
\begin{rem}
Any orbifold groupoid is  Morita equivalent to an action groupoid $M\rtimes \Gamma$, where $M$ is a smooth manifold, and $\Gamma$ is a compact Lie group acting smoothly on $M$ with its istropy groups being finite.
\end{rem}
\begin{rem}
Let $\sf{X}$ $= [M\rtimes \Gamma]$ be an orbifold. Then the homotopy type of the Borel construction $M\times _{\Gamma} E\Gamma $ is an orbifold invariant, namely it is invariant under Morita equivalences.
\end{rem}
\subsection*{Orbifold loop product}
In this section, we review the construction of orbifold loop product defined in \cite{LUX}.

The following notion of loop orbifold is defined also in \cite{LUX}.
\begin{defn}[\cite{LUX}]\label{defn}
Let $[M/G]$ be a global quotient orbifold. The \emph{loop orbifold} of $[M/G]$ is the action groupoid $P_{G}M \rtimes G $, where
\begin{align*}
P_{G}M & = \displaystyle \coprod _{g \in G} P_{g}M ,\\
P_{g}M & = \{\sigma :[0,1] \longrightarrow M | \sigma (1) = \sigma (0)g \} \subset LM  ,
\end{align*}
and $G$ acts on $P_{G}M$ as
\[
P_{g}M \times G \ni (\sigma , h) \mapsto \sigma h \in P_{h^{-1}gh}M  .
\]
\end{defn}
In the same paper, they prove the weak homotopy equivalence $L(X\times_{G} EG) \simeq P_{G}X \times _{G} EG$. Hence we can consider $P_{G}M \times _{G} EG$ instead of $L(M\times_{G} EG)$ for studying the string topology of the orbifold $[M/G]$ because of Whitehead theorem. They construct the orbifold loop product as follows. 
\begin{proof}[Construction of the orbifold loop product]
We consider the following pullback diagram
\begin{equation}
\xymatrix{
P_{g}M\times_{M}P_{h}M \ar[r]^{\tilde{\iota}} \ar[d]^{ev_{1/2}} & P_{g}M\times P_{h}M \ar[d]^{ev_{1}\times ev_{0}} \\
M \ar[r]^{\iota} & M\times M ,
}
\end{equation}
for any $g,h \in G$, where $ev_{t}$ denote the evaluation map. Then we have the generalized Pontryagin Thom map
\[
\tilde{\iota}_{!} : P_{g}M\times P_{h}M \longrightarrow (P_{g}M\times_{M}P_{h}M)^{ev_{1/2}^{*}\nu_{\iota}}  ,
\]
where $\nu_{\iota}$ denotes the normal bundle of the embbeding $\iota$, and $(P_{g}M\times_{M}P_{h}M)^{ev_{1/2}^{*}\nu_{\iota}}$ denotes the Thom space of the vector bundle $ev_{1/2}^{*}\nu_{\iota}$. We also have the concatenation map  $\gamma : P_{g}M \times_{M}P_{h}M \longrightarrow P_{gh}M $ with
\begin{equation}
\gamma (\sigma_{g}, \sigma_{h}) = 
\begin{cases}
	\sigma_{g}(2t) & (0 \leq t \leq 1/2) ,\\
	\sigma_{h}(2t-1) & (1/2 \leq t \leq 1) .
	\end{cases}
\end{equation}
Then we obtain a sequence of compositions
\begin{align*}
H_{p}(P_{g}M)\otimes H_{q}(P_{h}M) \xrightarrow{\times} \  H_{p+q}(P_{g}M\times P_{h}M) \xrightarrow{\tilde{\iota}_{!*}} \  H_{p+q}((P_{g}M\times_{M} P_{h}M)^{ev_{1/2}^{*}\nu_{\iota}}) \\
 \xrightarrow{Thom \ isom.}  \  H_{p+q-n}(P_{g}M\times_{M}P_{h}M) \xrightarrow{\gamma_{*}} \  H_{p+q-n}(P_{gh}M)  ,
\end{align*}
and we obtain a map 
\begin{equation}
\bullet : H_{p}(P_{g}M)\otimes H_{q}(P_{h}M) \longrightarrow H_{p+q-n}(P_{gh}M).
\end{equation}
By taking summation over $g,h \in G$, we obtain a map which we also denote by $\bullet$
\begin{equation}
\bullet : H_{p}(P_{G}M)\otimes H_{q}(P_{G}M) \longrightarrow H_{p+q-n}(P_{G}M).
\end{equation}
To lift up the map $\bullet$ to
\begin{equation}
\circ : H_{p}(P_{G}M\times_{G}EG)\otimes H_{q}(P_{G}M\times_{G}EG) \longrightarrow H_{p+q-n}(P_{G}M\times_{G}EG), 
\end{equation}
we use the following fundamental lemma in algebraic topology.
\begin{lem} \label{prp1}
	Let $p : \tilde{X} \to X$ be a finite galois covering and $G$ be its galois group. If $k$ is a field whose character is coprime to $|G|$, there exists an injective homomorphism $\mu_{*}$ called the transfer map
	\begin{equation}\label{eq1}
	\mu_{*} : H_{*}(X ;k) \longrightarrow H_{*}(\tilde{X} ;k)  ,
	\end{equation}
such that
\begin{equation}
\mathrm{Im}  \mu_{*} \cong H_{*}(\tilde{X} ;k)^{G}  ,
\end{equation}
	where the right hand side is the $G$-invariant subspace of $H_{*}(X ;k)$.
\end{lem}
\begin{proof}
We first define the transfer map $\mu_{*}: H_{*}(X) \to H_{*}(\tilde{X})$. Let $\mu$ be a homomorphism between singular chain complexes $\mu : C_{*}(X) \to C_{*}(\tilde{X}) $, with
\begin{align*}
\mu (\Delta)  = \sum_{lifts \ of \Delta} \tilde{\Delta}.
\end{align*}
The summation runs over all lifts of the singular simplex $\Delta$. We can see that $\mu$ commutes with the differentials, hence we can define $\mu_{*} : H_{*}(X) \to H_{*}(\tilde{X})  $, with
\begin{equation}
\mu_{*}(x) = \sum_{g\in G} g_{*}x  .
\end{equation}
Then both $p_{*}\circ \mu_{*} : H_{*}(X) \to H_{*}(X)$ and $\mu_{*}\circ p_{*} : H_{*}(\tilde{X})^{G} \to H_{*}(\tilde{X})$ are the multiplication by $|G| \in k$. As the characteristic of $k$ is coprime to $|G|$, the multiplication by  $|G|$ is an isomorphisms of vector spaces, which implies $\mu_{*}$ and $p_{*}$ are also isomorphisms. Thus (\ref{eq1}) holds.
\end{proof}
By using this transfer map, we obtain a sequence of compositions
\begin{align*}
H_{p}(P_{G}M\times_{G}EG)\otimes H_{q}(P_{G}M\times_{G}EG) \xrightarrow{\mu_{*}\otimes \mu_{*}} \  H_{p}(P_{G}M)\otimes H_{q}(P_{G}M) \\
 \xrightarrow{\bullet} \  H_{p+q-n}(P_{G}M) \xrightarrow{p_{*}} \  H_{p+q-n}(P_{G}M\times_{G}EG)  ,
\end{align*}
where $p_{*}$ is the map induced from the covering map $ p : P_{G}M\times EG \longrightarrow P_{G}M\times_{G}EG $.

The obtained map 
\begin{equation}
\circ : H_{p}(P_{G}M\times_{G}EG)\otimes H_{q}(P_{G}M\times_{G}EG) \longrightarrow H_{p+q-n}(P_{G}M\times_{G}EG), 
\end{equation}
is the desired orbifold loop product, which coincides with the ordinary loop product when $G$ is a trivial group by the above construction.
\end{proof}
\begin{rem}
It is shown in \cite{LUX} that the orbifold loop product $\circ$ is indeed an orbifold invariant.
\end{rem}
\section{Propositions for the proof of Theorem}
In this section, we prove some propositions that we use for the proof of main theorem. Unless otherwise stated, we denote by $M$ a smooth colosed oriented manifold, by $\mathcal{G}$ a path connected group acting continuously on $M$, and by $G$ a finite subgroup of $\mathcal{G}$ acting smoothly on $M$. 

The following proposition is proved in \cite{LUX} for computing loop product of some lens spaces. We prove it below because we use it in this paper.
\begin{prp}[\cite{LUX}] \label{prp4}
	Let $\mathcal{G}$ be a path connected group acting continuously on $M$, and $G$ a finite subgroup of $\mathcal{G}$ acting smoothly on $M$. Then for each $g \in G$, there exists a homotopy equivalence
	\[
	P_{g}X \simeq LM  .
	\]
\end{prp}
\begin{proof}
By the assumption on $G$, there exists a path $\theta_{g}$ in $\mathcal{G}$ which starts at the unit in $G$ and ends at $g$ for each $g \in G$. We fix such paths $\{\theta_{g}\}_{g\in G}$.

We consider the maps $\tau_{g} : P_{g}M \longrightarrow LM $, and $ \eta_{g} : LM \longrightarrow P_{g}M $, with for each $g\in G$, 
	\begin{equation}
	\tau_{g} (t) =  
	\begin{cases}
	\sigma(2t) & (0 \leq t \leq 1/2) ,\\
	\sigma (1)\check{\theta_{g}}(2t-1) & (1/2 \leq t \leq 1) ,
	\end{cases}
	\end{equation}
	and
	\begin{equation}
	\eta_{g}(l) =  
	\begin{cases}
	l(2t) & (0 \leq t \leq 1/2) ,\\
	l(1)\theta_{g}(2t-1) & (1/2 \leq t \leq 1) ,
	\end{cases}
	\end{equation}
	where $\check{\theta_{g}}$ denotes the image of $\theta_{g}$ by the map $\mathcal{G} \to \mathcal{G}$ wihch sends an element to its inverse. Then we can see $\tau_{g} \circ \eta_{g} \simeq id_{LM}$ and $\eta_{g} \circ \tau_{g} \simeq id_{P_{g}M}$, hence $P_{g}M \simeq LM$ for each $g\in G$.
\end{proof}
We consider the map $\Phi : \Omega\mathcal{G} \times LM \longrightarrow LM $ with
\begin{equation}\label{nat1}
\Phi (a,l) = (t \mapsto a(t)\cdot l(t)),
\end{equation}
and the induced map $\Phi_{*} : H_{*}(\Omega \mathcal{G})\otimes H_{*}(LM) \longrightarrow H_{*}(LM)$, which defines an action of the Pontrjagin ring $H_{*}(\Omega \mathcal{G})$ on $H_{*}(LM)$. If the action $\Phi_{*}$ of $H_{0}(\Omega \mathcal{G})$ on $H_{*}(LM)$satisfies $\Phi_{*}(x) = x$, for any $x \in H_{*}(LM)$, we call the action $\Phi_{*}$ \emph{trivial}.
\begin{prp}\label{prp}
Assume that the action $\Phi_{*}$ is trivial. Then the homotopy equivalence of Proposition \ref{prp4}
	\begin{equation}
	\displaystyle \coprod _{g \in G} P_{g}M \simeq \displaystyle \coprod _{g \in G}LM
	\end{equation}
	is $G$-equivariant at homology level with coefficients in $k$.
\end{prp}
We use the following lemma for the proof.
\begin{lem} \label{prp2}
Let $\mathcal{G}$ be a path connected group acting continuously on $M$, and $G$ a finite subgroup of $\mathcal{G}$ acting smoothly on $M$. Then $G$ acts on $H_{*}(LM)$ trivially.
\end{lem}
\begin{proof}
	The paths $\{\theta_{g}\}_{g\in G} $ in the proof of Proposition \ref{prp4} make a homotopy between the actions of $e$ and $g$ on $H_{*}(LM)$. 
	Because $H_{*}(LM)$ is discrete, they act trivially.
\end{proof}
\begin{proof}[Proof of Proposition \ref{prp}]	
	We should show the commutativity of the following diagram for each $g, h \in G$
	\begin{equation*}
	\xymatrix{
	H_{*}(LM) \ar[r]^{h_{*}} &  H_{*}(LM) \\
	H_{*}(P_{g}M) \ar[r]^{h_{*}} \ar[u]^{\tau_{g*}} & H_{*}(P_{h^{-1}gh}M) \ar[u]^{\tau_{h^{-1}gh*}}  .
	}
	\end{equation*}
	By Lemma \ref{prp2}, the upper $h_{*}$ is equal to an identity.
	Thus we should show $(\tau_{h^{-1}gh} \circ h \circ \eta_{g})_{*} = id$. 
	Let $a$ be the loop $h^{-1}\theta_{g}h * \check{\theta}_{h^{-1}gh}$ in $\mathcal{G}$ (here $*$ denotes concatenating operation defined by (\ref{con1}). See Figure \ref{Figure1}). Then the map $\tau_{h^{-1}gh} \circ h \circ \eta_{g}$ is equal to the map $\Phi (a, -) : LM \longrightarrow LM $, and the induced homomorphism $(\tau_{h^{-1}gh} \circ h \circ \eta_{g})_{*}$ is equal to the action $\Phi_{*}([a], -) : H_{*}(LM) \longrightarrow H_{*}(LM)$, which is trivial by the assumption.
\end{proof}
\begin{figure}
        \centering
        \def \svgwidth{60mm}
        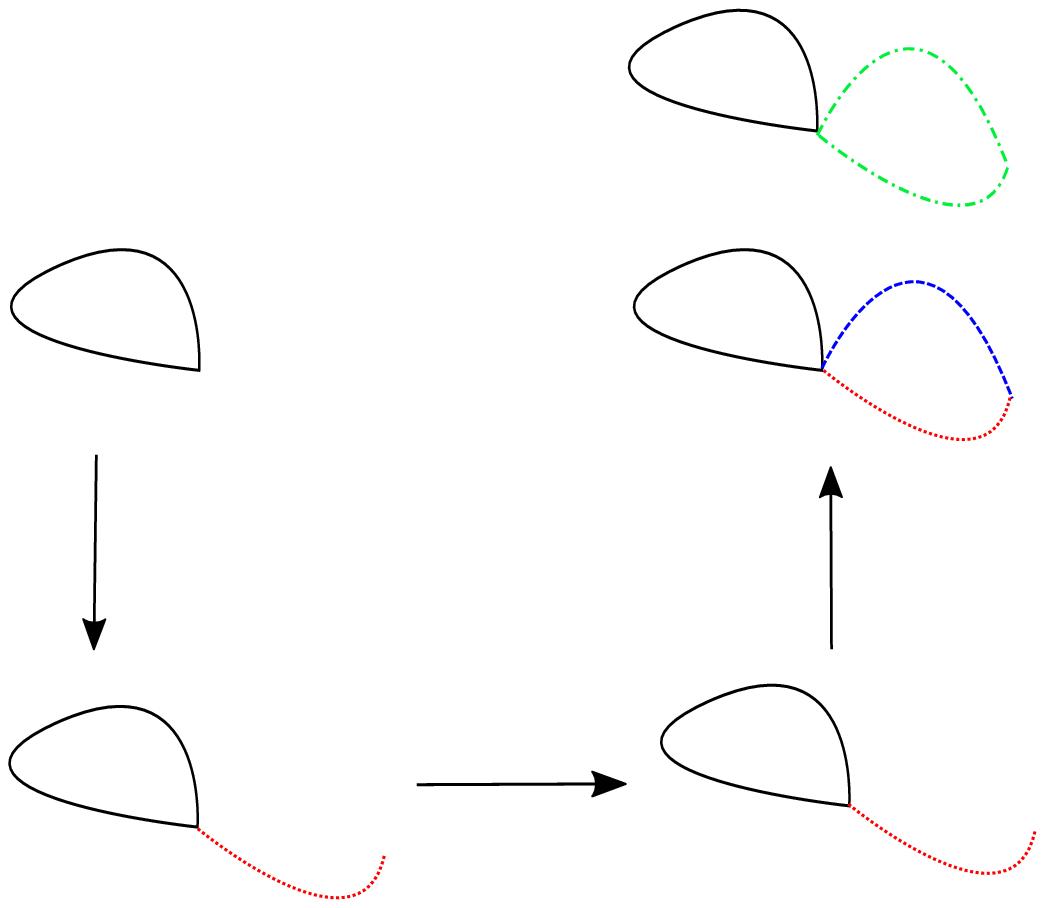
	\caption{for Proposition \ref{prp}}
        \label{Figure1}
\end{figure}
\section{Main theorem}
Let $\tau$ be a direct sum of $\tau_{g}$'s over $g \in G$, namely  $\tau = \oplus_{g} \tau_{g} : \displaystyle \coprod _{g \in G} P_{g}M \to \displaystyle \coprod _{g \in G}LM$. Then by Proposition \ref{prp} and  Lemma \ref{prp2}, $\tau$ induces an isomorphism
\begin{eqnarray}
		H_{*}(L[M/G]) &= & H_{*}\left(\displaystyle \coprod _{g \in G} P_{g}M \right)^{G} \\
		& \overset{\tau _{*}}{\cong} & H_{*}\left(\displaystyle \coprod_{g \in G} LM \right)^{G} \\
		&\cong & H_{*}(LM \times G)^{G} \\
		&\cong & (H_{*}(LM) \otimes H_{*}(G))^{G} \\
		&\cong & H_{*}(LM) \otimes H_{*}(G)^{G} \\
		&\cong & H_{*}(LM) \otimes Z(k[G]).
	\end{eqnarray}
We show that $\tau_{*}$ preserves loop products. Our main theorem is the following.
	\begin{thm}\label{thm}
	Let $\mathcal{G}$ be a path connected group acting continuously on $M$, and $G$ be a finite subgroup of $\mathcal{G}$ acting smoothly on $M$, and $k$ be a field whose characteristic is coprime to $|G|$.
	If the action $\Phi_{*}$ is trivial with coefficient in $k$, then there exists an isomorphism as $k$-algebras
	\begin{equation}\label{iso}
	\mathbb{H}_{*}(L[M/G];k) \cong \mathbb{H}_{*}(LM;k) \otimes Z(k[G])  ,
	\end{equation}
	where $Z(k[G])$ denotes the center of the group ring $k[G]$.
	\end{thm}
	For the proof, we should show that the diagram
\begin{equation} \label{diagram1}
	\xymatrix{
	H_{p}\left( \displaystyle \coprod_{g} LM \right)^{G} \otimes H_{q}\left(\displaystyle \coprod_{g} LM\right)^{G} \ar[r]^(0.58){\circ} &  H_{p+q-n}\left(\displaystyle \coprod_{g} LM\right)^{G} \\
	H_{p}\left(\displaystyle \coprod_{g}P_{g}M\right)^{G} \otimes H_{q}\left(\displaystyle \coprod_{g} P_{g}M\right)^{G} \ar[r]^(0.58){\circ} \ar[u]^{\tau_{*}\otimes \tau_{*}} & H_{p+q-n}\left(\displaystyle \coprod_{g} P_{g}M\right)^{G} \ar[u]^{\tau_{*}}  ,
	}
	\end{equation}
is commutative. To prove this we need some lemmas. We consider the following commutative diagram of pullbacks
	\begin{equation}\label{pb}
	\xymatrix{
	LM\times_{M} LM \ar@/^60pt/[dddd]^(0.2){ev_{1/2}} \ar[rr] &  & LM\times LM \ar@/^50pt/[dddd]^(0.2){ev_{1/2}\times ev_{0}} \\
	  &  &  \\
	P_{g}M\times_{M} P_{h}M \ar[dd]^{ev_{1/2}} \ar@{-->}[uu]^{\tau_{1/2}}  \ar[rr]|\hole &  & P_{g}M\times P_{h}M \ar[uu]^{\tau_{g}\times \tau_{h}} \ar[dd]^{ev_{1}\times ev_{0}} \\
	  &  &  \\
	M \ar[rr]^{\Delta} &  & M\times M  ,
	}
	\end{equation}
where $\tau_{1/2}$ is a map induced from the universality of the pullback.
\begin{lem}\label{l}
\begin{equation}
	\xymatrix{
	LM  &  &  LM\times_{M} LM \ar[ll]^{\gamma^{\prime}} \\
	  &  &  \\
	P_{gh}M \ar[uu]^{\tau_{gh}} &  &  P_{g}M\times_{M}P_{h}M \ar[ll]^{\gamma} \ar@{-->}[uu]^{\tau_{1/2}}  ,
	}
	\end{equation}
	is commutative at homology level, where $\gamma^{\prime}$ and $\gamma$ are concatenating maps defined by
	\begin{equation}
	\gamma^{\prime}((l_{1},l_{2}))(t) =  
	\begin{cases}
	l_{1}(2t) & (0 \leq t \leq 1/4) ,\\
	l_{2}(2t-1/2) & (1/4 \leq t \leq 3/4) ,\\
	l_{1}(2t-1) & (3/4 \leq t \leq 1) ,
	\end{cases}
	\end{equation}
	and
	\begin{equation}
	\gamma ((\sigma_{1},\sigma_{2}))(t) =
	\begin{cases}
	\sigma_{1}(2t) & (0 \leq t \leq 1/2) ,\\
	\sigma_{2}(2t-1) & (1/2 \leq t \leq 1)	  .
	\end{cases}	
	\end{equation}
\end{lem}
\begin{proof}
	Let $b$ be the loop $h\check{\theta}_{h}*\check{\theta}_{g}*g^{-1}\theta_{gh}$ in $\mathcal{G}$ (See Figure \ref{Figure2}). Then we have the commutative diagram
	\begin{equation}
	\xymatrix{
	LM \ar[dd]^{\eta_{gh}} &  &  LM\times_{M} LM \ar[ll]^{\gamma '} \\
	  &  &  \\
	P_{gh}M &  &  P_{g}M\times_{M}P_{h}M \ar@{-->}[uu]^{\tau_{1/2}} \ar[ddll]^{\gamma} \\
	& & \\
	P_{gh}M \ar[uu]^{* b} & & ,
	}
	\end{equation}
	where $* b$ denotes the map $\sigma \mapsto \sigma * \sigma(1)b$. Hence it reduces to show that $(* b)_{*} : H_{*}(P_{gh}M) \to H_{*}(P_{gh}M)$ is equal to an identity, which follows from the similar argument in the proof of Proposition \ref{prp}.
	\end{proof}
\begin{figure}
        \centering
        \def \svgwidth{60mm}
        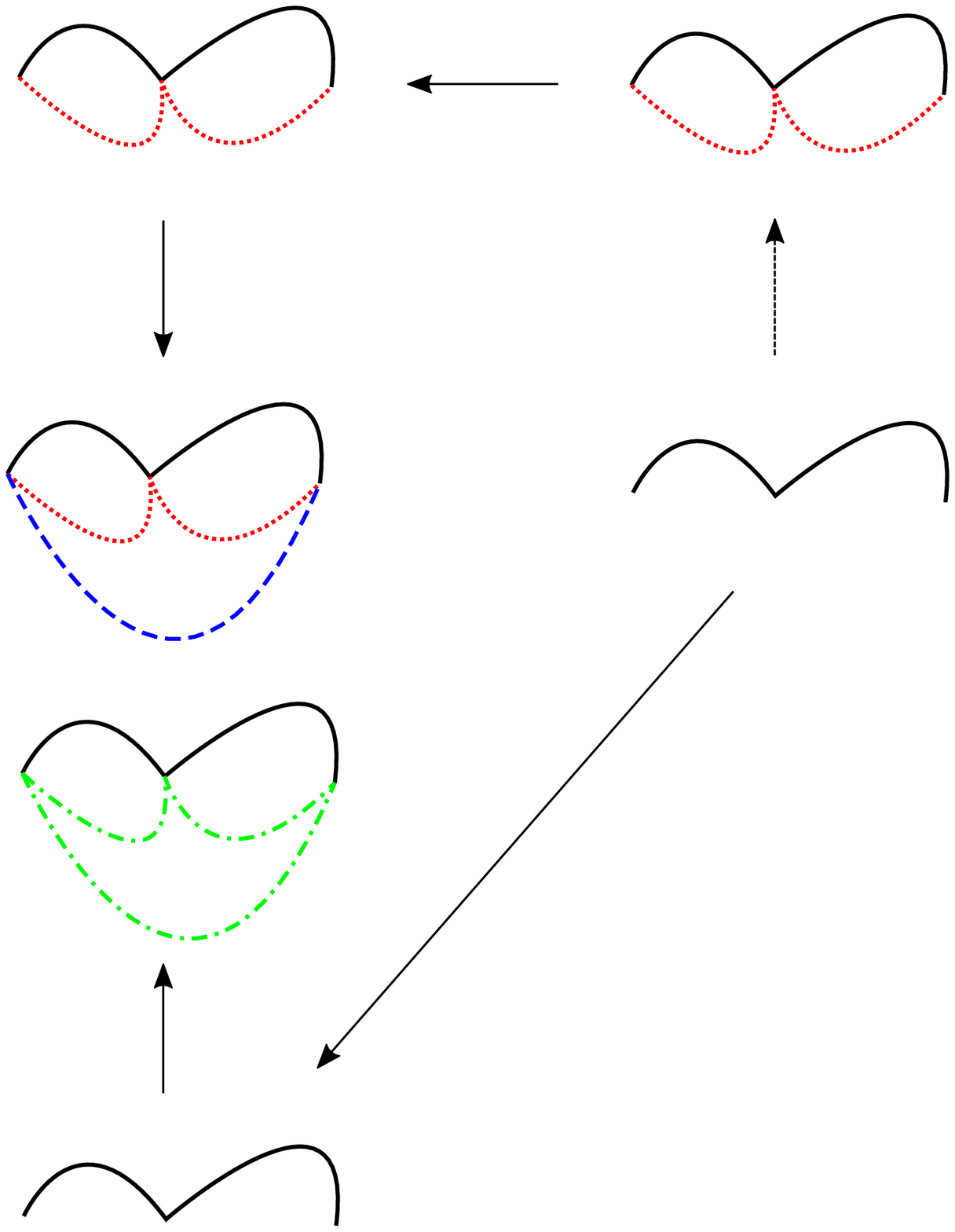
	\caption{for Lemma \ref{l}}
        \label{Figure2}
\end{figure}
\begin{lem}\label{lem5}
The concatenating map $\gamma^{\prime}$ defines the same product as the loop product $\circ$. namely the diagram
	\begin{equation}
	\xymatrix{
	H_{*}(LM\times LM) \ar[r]^(0.4){P.T.} & H_{*}((LM\times_{M} LM)^{TM}) \ar[r]^{Thom.} & H_{*\dim M}(LM\times_{M} LM) \ar[r]^(0.6){\gamma^{\prime}_{*}} & H_{*\dim M}(LM)  \\
	H_{*}(LM\times LM) \ar[u]^{(\rho_{1/2}\times id)_{*}} \ar[r]^(0.4){P.T.} & H_{*}((LM\times_{M} LM)^{TM}) \ar[u]^{\widetilde{(\rho_{1/2}\times_{1/2} id)}_{*}} \ar[r]^{Thom.} & H_{*\dim M}(LM\times_{M} LM) \ar[u]^{(\rho_{1/2}\times_{1/2} id)_{*}}\ar[r]^(0.6){\gamma_{*}} & H_{*\dim M}(LM) \ar[u]^{\rho_{1/4*}}
	}
	\end{equation}
	commutes.
	\end{lem}
	\begin{proof}
	It follows from the commutative diagram
	\begin{equation}
	\xymatrix{
	LM & & LM\times_{M} LM \ar[ll]^{\gamma^{\prime}} \ar@/^55pt/[dddd] \ar[rr] &  & LM\times LM \ar@/^50pt/[dddd]^(0.2){ev_{1/2}\times ev_{0}} \\
	&  &  &  & \\
	LM \ar[uu]^{\rho_{1/4}} & & LM\times_{M} LM \ar[ll]^{\gamma} \ar[dd] \ar@{-->}[uu]^{\rho_{1/2}\times_{1/2} id} \ar[rr]|\hole &  & LM\times LM \ar[uu]^{\rho_{1/2}\times id} \ar[dd]^{ev_{0}\times ev_{0}} \\
	&  & &  &  \\
	 & & M \ar[rr]^{\Delta} &  & M\times M  ,
	}
	\end{equation}
	and the naturality of Pontrjagin-Thom construction and Thom isomorphism, where $\rho_{a}$ is the parameter transformation map defined by $\rho_{a}(l) = (t \mapsto l(t+a))$.
	\end{proof}
\begin{prp} \label{prp5}
	For any $g,h \in G$, the diagram
	\begin{equation} \label{diagram2}
	\xymatrix{
	H_{p}(LM) \otimes H_{q}(LM)\ar[r]^(0.57){\circ} &  H_{p+q-n}(LM) \\
	H_{p}(P_{g}M) \otimes H_{q}(P_{h}M) \ar[r]^(0.57){\circ} \ar[u]^{\tau_{g*}\otimes \tau_{h*}} & H_{p+q-n}(P_{gh}M) \ar[u]^{\tau_{gh*}}
	}
	\end{equation}
	commutes.
	\end{prp}
	\begin{proof}
	Because of the diagram (\ref{pb}) and Lemma \ref{l}, by the naturality of Pontrjagin-Thom map and Thom isomorphism, we obtain the following commutative diagram
	\begin{equation}
	\xymatrix{
	H_{*}(LM\times LM) \ar[r] & H_{*}((LM\times_{M} LM)^{TM}) \ar[r] & H_{*\dim M}(LM\times_{M} LM) \ar[r] & H_{*\dim M}(LM)  \\
	H_{*}(P_{g}M\times P_{h}M) \ar[u]^{(\tau_{g}\times \tau_{h})_{*}} \ar[r] & H_{*}((P_{g}M\times_{M} P_{h}M)^{TM}) \ar[u]^{\widetilde{\tau_{1/2*}}} \ar[r] & H_{*\dim M}(P_{g}M\times_{M} P_{h}M) \ar[u]^{\tau_{1/2*}} \ar[r] & H_{*\dim M}(P_{gh}M) \ar[u]^{\tau_{gh*}}  ,
	}
	\end{equation}
	which implies the diagram (\ref{diagram2}) is commutative by Lemma \ref{lem5}.
	\end{proof}
\begin{proof}[Proof of Thoerem \ref{thm}]
	By the definition of the loop product, diagram (\ref{diagram1}) is same as the following diagram
	\begin{equation}
	\xymatrix{
	H_{p}\left(\displaystyle \coprod_{g} P_{g}M\right)^{G}\otimes H_{q}\left(\displaystyle \coprod_{g} P_{g}M\right)^{G}  \ar@{^{(}->}[d] \ar[r]^{\tau_{*}\otimes \tau_{*}} &  H_{p}\left(\displaystyle \coprod_{g} LM\right)^{G}\otimes H_{q}\left(\displaystyle \coprod_{g} LM\right)^{G}  \ar@{^{(}->}[d]\\
H_{p}\left(\displaystyle \coprod_{g} P_{g}M\right)\otimes H_{q}\left(\displaystyle \coprod_{g} P_{g}M\right) \ar[r]^{\tau_{*}\otimes \tau_{*}} \ar[d]^(0.58){\circ} & H_{p}\left(\displaystyle \coprod_{g} LM\right)\otimes H_{q}\left(\displaystyle \coprod_{g} LM\right) \ar[d]^(0.58){\circ} \\
H_{p+q-\dim M}\left(\displaystyle \coprod_{g} P_{g}M\right) \ar[r]^{\tau_{*}} \ar[d]^{\sigma_{*}} & H_{p+q-\dim M}\left(\displaystyle \coprod_{g} LM\right) \ar[d]^{\sigma_{*}} \\
H_{p+q-\dim M}\left(\displaystyle \coprod_{g} P_{g}M\right)^{G} \ar[r]^{\tau_{*}} & H_{p+q-\dim M}\left(\displaystyle \coprod_{g} LM\right)^{G}  ,
	 }
	\end{equation}	
	where $\sigma_{*}$ denotes composition of projection and isomorphism $\mu_{*}$ defined in the proof of Proposition \ref{prp1},
	\begin{equation}
	\xymatrix{
	H_{*}(LM\times G) \ar[r]^(0.45){pr_{*}} & H_{*}((LM\times G)/G) \ar[r]^{\mu_{*}}_{\cong} & H_{*}(LM\times G)^{G}  ,
	}
	\end{equation}
	and
	\begin{equation}
	\xymatrix{
	H_{*}\left(\displaystyle \coprod_{g}P_{g}M\right) \ar[r]^(0.45){pr_{*}} & H_{*}(\left (\displaystyle \coprod_{g}P_{g}M\right )/G) \ar[r]^{\mu_{*}}_{\cong} & H_{*}\left(\displaystyle \coprod_{g}P_{g}M\right)^{G}  .
	}
	\end{equation}
	The $G$-equivariance of $\tau_{*}$ implies the commutativity of the upper square. And Proposition \ref{prp5} implies the commutativity of the middle square. The commutativity of the lower square follows from the definition of $\mu_{*}$ and $G$-equivariance of $\tau_{*}$. 
\end{proof}
	\section{Examples}
	In this section, we will compute some examples of orbifold loop homology. Before that, we will proof some propositions which is useful to apply our main theorem and to compute examles.
	\begin{prp}
	If $\mathcal{G}$ is simply connected, then the action $\Phi_{*}$ is trivial for any field $k$.
	\end{prp}
	\begin{proof}
	Because $H_{0}(\Omega \mathcal{G})$ is spanned by only the class represented by the identity loop, its action on $H_{*}(LM)$ is trivial.
	\end{proof}
	\begin{prp} \label{ex2}
	If $H_{n}(LM;k) = k$ for a field $k$, then the action $\Phi_{*}$ is trivial for any pair $(\mathcal{G},G)$.
	\end{prp}
To prove this proposition, we use the following lemma shown by Hepworth.
\begin{lem}[\cite{He2}]\label{act}
Let $\Phi$ be the natural map defined by (\ref{nat1}). Then the induced linear action $\Phi_{*} : H_{0}(\Omega \mathcal{G})\otimes H_{*}(LM) \longrightarrow H_{*}(LM)$ is an algebra action, namely for any $\alpha \in H_{0}(\Omega \mathcal{G})$ and any $x,y \in H_{*}(LM)$, they satisfy
\begin{equation}
\alpha \cdot (x \circ y) = (\alpha \cdot x)\circ y = x\circ (\alpha \cdot y) .
\end{equation}
\end{lem}
	\begin{proof}
	Let $a$ be a loop in $\mathcal{G}$ which starts and ends at the unit. We should show that the action $\Phi_{a*} = \Phi_{*}([a], -) : H_{*}(LM) \longrightarrow H_{*}(LM) $ is trivial, where $\Phi_{a}$ denotes the map $\Phi (a, -)$. We also denote by $\Phi_{a}$ the restriction map of the map $\Phi (a, -) : LM \longrightarrow LM$ to $M \subset LM$, where we regard $M$ as the image of the map $ c : M \longrightarrow LM $ which assigns the constant loop.
	
We consider the  sequence of maps
	\begin{equation}
	\xymatrix{
	  &  &  \\
		H_{\dim M}(M) \ar@<0.5ex>[r]^{\Phi_{a*}} \ar@<-0.5ex>[r]_{c_{*}} \ar@/^30pt/[rr]^{id_{*}} \ar@{}[urr]|{\circlearrowright} & H_{\dim M}(LM) \ar[r]^{ev_{*}} & H_{\dim M}(M) .
	}
	\end{equation}
	As $H_{\dim M}(LM) = H_{\dim M}(M) = k$, we obtain $\Phi_{a*} = c_{*} = id$. Hence for any $x \in H_{*}(LM)$, we have the equality by Lemma \ref{act}
	\begin{eqnarray*}
	\Phi_{a*}x & = & \Phi_{a*}([M] \circ x) \\
	& = & (\Phi_{a*} [M]) \circ x \\
	& = & x.
	\end{eqnarray*}
	Therefore the action of $H_{0}(\Omega \mathcal{G})$ on $H_{*}(LM)$ is trivial.
	\end{proof} 
	\begin{prp}\label{ex1}
	Let $k$ be a field. If the conditions
	\begin{itemize}
	\item[(i)] $\pi_{1}M = 1$,
	\item[(ii)] $|\pi_{1}\mathcal{G}| =: r < \infty$,
	\item[(iii)] the homomorphism $H_{*}(\Omega M;k) \longrightarrow H_{*}(LM;k)$ induced by the inclusion $\Omega M \longrightarrow LM$ is injective, namely the free loop fibration $\Omega M \to LM \to M$ is Totally Non-Cohomologous to Zero (TNCZ) with respect to the field $k$,
	\item[(iv)] the characteristic of $k$ is coprime to $r$,
	\end{itemize}
	are satisfied, then the action $\Phi_{*}$ is trivial.
	\end{prp}
	For the proof of Proposition \ref{ex1}, we use the following lemma shown by Hepworth.
	\begin{lem}[\cite{He2}]\label{rng}
	Let $\mathcal{G}$ be a topological group acting continuously on a closed oriented manifold $M$. Then the  homomorphism $\Phi _{*}': H_{*}(\Omega \mathcal{G} ;k) \longrightarrow H_{*}(LM ;k) $ with
	\begin{equation}
	\Phi_{*}'[a] = [a]\cdot [M] ,
	\end{equation}
	 commutes with the products, namely they satisfy
\begin{equation}
\Phi_{*}'[a]\circ \Phi_{*}'[a'] = \Phi_{*}'([a]\cdot [a']),
\end{equation}
where the product in the right hand side denotes the Pontrjagin product.
	\end{lem}
	\begin{proof}[Proof of Proposition \ref{ex1}]
	We denote below $\dim M$ by $n$. Let $a$ be a loop in $\mathcal{G}$ which starts and ends at the unit. By the arguments in the proof of Proposition \ref{ex2}, we should show that  $\Phi_{a*}[M]$ is the unit in $\mathbb{H}_{0}(LM)$ for the triviality of the action of $H_{0}(\Omega \mathcal{G})$.	
	By the TNCZ assumption, we have the ring isomorphism
	\begin{equation}
	\mathbb{H}_{*}(LM) \cong H_{*}(\Omega M) \otimes \mathbb{H}_{*}(M)  ,
	\end{equation}
	hence we obtain the linear isomorphism
	\begin{equation}
	H_{n}(LM) \cong \displaystyle \bigoplus _{p+q=n} H_{p}(\Omega M) \otimes H_{q}(M).
	\end{equation}
	By the assumption of $\pi_{1} M =1$, we have $H_{0}(\Omega M) \cong k$, hence $H_{0}(\Omega M)\otimes H_{n}(M)$ is a $1$ dimensional vector space, and we can put 
	\begin{equation}
	\Phi _{a*}[M] = x\cdot 1 + \displaystyle \sum _{i} y_{i}\cdot \sigma_{i} \otimes \delta_{i}, 
	\end{equation}
	where
	\begin{align*}
	1 = 1\otimes [M] \in H_{0}(\Omega M)\otimes H_{n}(M) ,\\
	\sigma_{i} \in H_{>0}(\Omega M) ,\\
	\delta_{i} \in H_{<n}(M) ,\\
	deg \sigma_{i} + deg \delta_{i} = n  ,\\
	x,y_{i} \in k  .
	\end{align*}
	We can see $x = 1$ as follows. We have a ring homomorphism induced by the evaluation map $ev_{*} : \mathbb{H}_{*}(LM) \longrightarrow \mathbb{H}_{*}(M) $. Hence we have 
	\begin{equation}
		ev_{*}(\sigma_{i}\otimes \delta_{i}) = ev_{*}(\sigma_{i}\otimes [M]) \circ ev_{*}(1\otimes \delta_{i})  .
	\end{equation}
	Because the deree of $\sigma_{i}$ is positive, the degree of $\sigma_{i} \otimes [M]$ is larger than $n$, hence $ev_{*}(\sigma_{*}\otimes [M])$ is equal to $0$. Therefore, we obtain 
	\begin{eqnarray*}
	ev_{*}\Phi _{a*}[M] & = & x\cdot ev_{*}1 + \displaystyle \sum _{i} y_{i}\cdot ev_{*}(\sigma_{i} \otimes \delta_{i})  \\
	& = &x\cdot ev_{*} 1 .
	\end{eqnarray*}
	Moreover, we have the diagram
	\begin{equation}
	\xymatrix{
	  &  &  \\
		H_{n}(M) \ar[r]^{\Phi_{a*}}  \ar@/^30pt/[rr]^{id_{*}} \ar@{}[urr]|{\circlearrowright} & H_{n}(LM) \ar[r]^{ev_{*}} & H_{n}(M) .
	}
	\end{equation}
	Thus we have 
	\begin{equation}
	 x\cdot [M] = x\cdot ev_{*}1 = ev_{*}\circ \Phi_{a*}[M] = [M] ,
	\end{equation}
	hence we obtain $x = 1$.
	Furthermore, we can show that $y_{i}$'s are all $0$ as follows. Because $\delta_{i}$'s are all nilpotent, we have 
	\begin{equation}
	(\Phi_{a*}[M] - 1)^{N} = 0  ,
	\end{equation} 
	for sufficiently large $N$. By the assumption of $|\pi_{1}\mathcal{G}| = r < \infty$ and by Lemma \ref{rng}, we have 
	\begin{equation}
	(\Phi_{a*}[M])^{r} = (\Phi_{*}'[a])^{r} = \Phi_{*}'[a]^{r} = \Phi_{*}'[1] =1.
	\end{equation}
	Since the characteristic of $k$ is coprime to $r$, the common divisor of $(\Phi_{a*}[M] - 1)^{N}$ and $(\Phi_{a*}[M])^{r} - 1$ is $\Phi_{a*}[M] - 1$. Hence we conclude that $y_{i}$'s are all 0, which implies $\Phi_{a*}[M] = 1$.
	\end{proof}
	\begin{exam}
	We consider the case $(M, \mathcal{G}, G, k) = (S^{3}, Spin(3), \pi, k)$, where $\pi = \langle a,b,c  |  a^{2} = b^{2} = c^{2} = (ab)^{2} = (bc)^{3} = (ca)^{5} \rangle$ and $k$ is an algebraic closed field whose characteristic of is not any 2,3, or 5. The finite group $\pi$ acts on $S^{3}$ and the quotient manifold $S^{3}/\pi$ is called the \emph{Poincar\'{e} homology sphere}.
	Then we have an algebra isomorphism
	\begin{eqnarray*}
		\mathbb{H}_{*}(L(S^{3}/\pi); k) &\cong & \mathbb{H}_{*}(S^{3}; k) \otimes Z(k[\pi]) \\
		& \cong & \mathbb{H}_{*}(S^{3}; k) \otimes k^{9}.
	\end{eqnarray*}
	Here we use the fact that if $k$ is an algebraic closed field and $G$ is a finite group, then $Z(k[G]) \cong k^{c(G)}$, where $c(G)$ denotes the number of conjugacy classes of $G$.
	\end{exam}
	\begin{rem}
	If $k = \mathbb{Q}$, Vigu\'{e} \cite{Vi} shows that the following are equivalent.
	\begin{itemize}
	\item[(i)] The homomorphism $H_{*}(\Omega M;k) \longrightarrow H_{*}(LM;k)$ induced by the inclusion $\Omega M \to LM$ is injective.
	\item[(ii)] $H^{*}(M;k)$ is a free graded commutative algebra.
	\end{itemize}
	\end{rem}
	\begin{rem}
	For $M = S^{n}, CP^{n}$, Menichi \cite{Me2} shows that the following are equivalent.
	\begin{itemize}
	\item[(i)] The homomorphism induced by the inclusion $H_{*}(\Omega M;k) \longrightarrow H_{*}(LM;k)$ is injective  .
	\item[(ii)] The Euler number $\chi (M) = 0$ in $k$  .
	\end{itemize}
	\end{rem}
	\begin{rem}
	The necessary condition for the free loop fibration
	\begin{equation}\label{fib}
	\Omega M \longrightarrow LM \longrightarrow M
	\end{equation}
	being the TNCZ fibration with respect to a finite field $\mathbb{F}_{p}$ is studied for many homogeneous spaces by Kuribayashi \cite{Kur}. For example, 
	\begin{itemize}
	\item For $M = CP^{n}, HP^{n}$, (\ref {fib}) is TNCZ if and only if  the Euler number $\chi(M) = 0$ in $\mathbb{F}_{p}$.
	\item For $M = SU(m+n)/SU(n), Sp(m+n)/Sp(n)$, (\ref {fib}) is TNCZ for any $m,n \geq 0, p>0$.
	\item For $SO(m+n)/SO(n)$ and $p > 2$, (\ref {fib}) is TNCZ if and only if $n$ is odd.
	\item For $SO(m+n)/SO(n)$ and $p=2$, (\ref {fib}) is TNCZ if $m \leq 4$ or $1\leq m \leq 8$ and $n \geq 43$.
	\item For $U(m+n)/U(m) \times U(n), Sp(m+n)/Sp(m) \times Sp(n
)$, (\ref {fib}) is \textit{not}  TNCZ for any $m,n \geq 0, p>0 $.
	\end{itemize}
	\end{rem}
%
%
%
%

\end{document}